\documentclass{siamart0216}
\usepackage{amsfonts}

\usepackage{amsfonts,amsmath,amssymb}
\usepackage{mathrsfs,mathtools,stmaryrd,wasysym}
\usepackage{enumerate}
\usepackage{xspace} 
\usepackage{esint}
\usepackage{graphicx}
\DeclareMathAlphabet{\mathpzc}{OT1}{pzc}{m}{it}

\newcommand{\R}{\mathbb{R}}
\newcommand{\calP}{\mathcal{P}}
\newcommand{\bu}{\mathbf{u}}
\newcommand{\bv}{\mathbf{v}}
\newcommand{\GRAD}{\nabla}
\DeclareMathOperator{\DIV}{div}

\newcommand{\ie}{i.e.,\@\xspace}

\newcommand{\calG}{{\mathcal{G}}}
\newcommand{\bF}{{\mathbf{F}}}
\newcommand{\bL}{{\mathbf{L}}}
\newcommand{\bW}{{\mathbf{W}}}

\newcommand{\bC}{{\mathbf{C}}}
\newcommand{\bG}{{\mathbf{G}}}
\newcommand{\bef}{{\mathbf{f}}}
\newcommand{\bn}{{\mathbf{n}}}

\newcommand{\bphi}{{\boldsymbol{\varphi}}}

\DeclareMathOperator{\supp}{supp}

\newsiamremark{remark}{Remark}

\DeclareMathOperator*{\dist}{dist}

\newcommand{\TheTitle}{The Poisson and Stokes problems in nonconvex, Lipschitz polytopes}
\newcommand{\ShortTitle}{The Poisson problem on weighted spaces}
\newcommand{\TheAuthors}{E.~Ot\'arola \and A.J.~Salgado}

\headers{\ShortTitle}{\TheAuthors}

\title{{\TheTitle}\thanks{EO has been supported in part by CONICYT through FONDECYT project 3160201. AJS has been supported in part by NSF grant DMS-1720213.}}

\author{
  Enrique Ot\'arola\thanks{Departamento de Matem\'atica, Universidad T\'ecnica Federico Santa Mar\'ia, Valpara\'iso, Chile.
    (\email{enrique.otarola@usm.cl}, \url{http://eotarola.mat.utfsm.cl/}).}
  \and
  Abner J.~Salgado\thanks{Department of Mathematics, University of Tennessee, Knoxville, TN 37996, USA.
    (\email{asalgad1@utk.edu}, \url{http://www.math.utk.edu/\string~abnersg})}
}

\ifpdf
\hypersetup{
  pdftitle={\TheTitle},
  pdfauthor={\TheAuthors}
}
\fi

\date{Draft version of \today.}

\begin{document}

\maketitle

\begin{abstract}
We show the well-posedness of the Poisson and Stokes problems in weighted spaces over nonconvex, Lipschitz polytopes. For a particular range of $p$, we consider those weights in the Muckenhoupt class $A_p$ that have no singularities in a neighborhood of the boundary of the domain.
\end{abstract}



\section{Introduction}
\label{sec:intro}

Let $d \in \{2,3\}$ and $\Omega$ be bounded polytope of $\R^d$ with Lipschitz boundary. Notice that we do not assume that $\Omega$ is convex. The purpose of this work is to study the well-posedness of the Dirichlet problem for the Poisson equation
\begin{equation}
\label{eq:PoissonStrong}
  -\Delta u = F \ \text{in } \Omega, \qquad u = 0 \ \text{on } \partial\Omega,
\end{equation}
and the Stokes problem
\begin{equation}
\label{eq:StokesStrong}
  -\Delta \bu + \GRAD \pi = - \DIV \bF, \ \DIV \bu = g, \ \text{in } \Omega, \qquad \bu = 0 \ \text{on } \partial\Omega,
\end{equation}
where we allow the data: $F$ and $(\bF,g)$, respectively to be singular.

The main technical tool that will allow us to assert certain degree of either regularity or integrability on the singular data and solutions, is the theory of weighted spaces \cite{MR0293384,MR1800316}. This has been carried out with a large degree of success for smooth domains. On the other hand, to the best of our knowledge, in the case of, possibly convex, polytopes very little has been done in this direction. For instance, \cite{MR2888310} proves a weighted Helmholtz decomposition on convex polytopes that is equivalent to the well-posedness of \eqref{eq:PoissonStrong}. However, as described in \cite{DO:17}, the argument presented there has a flaw. This was corrected in \cite{DO:17} for convex polytopes, and it is our intention here to, at least partially, remove the convexity assumption and study also the Stokes problem \eqref{eq:StokesStrong}. We will obtain well-posedness on weighted spaces, for a class of weights that do not have singularities or degeneracies near the boundary.

Our presentation will be organized as follows. Some preliminaries will be discussed in Section~\ref{sec:prelim}; where we will introduce the class of weights we shall operate with. The Poisson problem \eqref{eq:PoissonStrong} will be studied in Section~\ref{sec:main} along with some immediate applications of its well-posedness. Finally, the Stokes problem \eqref{eq:StokesStrong} will be analyzed in Section~\ref{sec:Stokes}. 

\section{Preliminaries}
\label{sec:prelim}

We will make repeated use of weighted Lebesgue and Sobolev spaces when the weight belongs to a Muckenhoupt class $A_p$. We refer the reader to \cite{MR1774162,NOS3,MR1800316,MR2491902} for the basic facts about Muckenhoupt classes and the ensuing weighted spaces. Here we only mention that a standard example of a Muckenhoupt weight is the distance to a lower dimensional object, see \cite{MR3215609}. In particular, if $z \in \Omega$ and we define the weight
\begin{equation}
\label{eq:defofourweight}
  \omega_z(x) = |x - z|^\alpha,
\end{equation}
then $\omega_z \in A_p$ provided that $\alpha \in (-d,d(p-1))$.

It is important to notice that in the example above, since $z \in \Omega$, there is a neighborhood of $\partial\Omega$ where the weight $\omega_z$ has no degeneracies or singularities. In fact, it is continuous and strictly positive. This observation allows us to define a restricted class of Muckenhoupt weights for which our results will hold. The following definition is motivated by \cite[Definition 2.5]{MR1601373}.

\begin{definition}[class $A_p(\Omega)$]
\label{def:ApOmega}
Let $\Omega \subset \R^d$ be a Lipschitz domain. For $p \in (1, \infty)$ we say that $\omega \in A_p$ belongs to $A_p(\Omega)$ if there is an open set $\calG \subset \Omega$, and positive constants $\varepsilon>0$ and $\omega_l>0$ such that:
\begin{enumerate}[1.]
  \item $\{ x \in \Omega: \dist(x,\partial\Omega)< \varepsilon\} \subset \calG$,
  
  \item $\omega \in C(\bar\calG)$, and
  
  \item $\omega_l \leq \omega(x)$ for all $x \in \bar\calG$. 
\end{enumerate}
\end{definition}

We shall also make use of the fact that if $p \in (1,\infty)$, $p'=p/(p-1)$ is its conjugate exponent, and $\omega \in A_p$, then $\omega':=\omega^{-p'/p} \in A_{p'}$ with $[\omega']_{A_{p'}} = [\omega]_{A_p}$, where we set 
\[
  [\omega]_{A_p} = \sup_B \left(\fint_B \omega\right)\left(\fint_B \omega'\right)^{p/p'}
\]
and the supremum is taken over all balls $B$.

The ideas we will use to prove our well-posedness results will, mainly, follow those used to prove \cite[Theorem 5.2]{MR1601373}. Essentially, owing to the fact that $\varpi$ is a regular function on a layer near the boundary of $\Omega$, we will use well-posedness on weighted spaces for smooth domains in the interior and an unweighted result near the boundary and then patch these together. To be able to separate these two pieces we define cutoff functions $\psi_i,\psi_\partial \in C_0^\infty(\R^d)$, $\psi_i + \psi_\partial \equiv 1$ on $\bar\Omega$ with the following properties:
\begin{enumerate}[$\bullet$]
  \item $\psi_i \equiv 1$ in a neighborhood of $\Omega \setminus \calG$,
  
  \item $\psi_i \equiv 0$ in a neighborhood of $\partial\Omega$, and
  
  \item setting $\Omega_i$ to be the interior of $\supp \psi_i$, then $\partial\Omega_i \in C^{1,1}$.
\end{enumerate}
Note that, without loss of generality, we can assume that $\partial\calG$ is Lipschitz. Observe also that $\supp \GRAD \psi_i \cup \supp \GRAD \psi_\partial \subset \bar \calG$.

For future use we define
\begin{equation}
\label{eq:defofp1}
  p_1 > \begin{dcases}
          3 & d = 3, \\
          4 & d = 2.
        \end{dcases}
  \qquad
  p_0 = p_1'.
\end{equation}

Finally, the relation $A \lesssim B$ will mean that $A \leq c B$ for a nonessential constant $c$ that might change at each occurrence.

\section{The Poisson problem}
\label{sec:main}

Let us now study problem \eqref{eq:PoissonStrong}. We begin by stating our definition of weak solution. Namely, for $p \in (1,\infty)$ and $\varpi \in A_p$, given $F \in W^{-1,p}(\varpi,\Omega)$ we seek for $u \in W^{1,p}_0(\varpi,\Omega)$ such that
\begin{equation}
\label{eq:Poisson}
  \int_\Omega \GRAD u \GRAD \varphi = \langle F, \varphi \rangle, \quad \forall \varphi \in C_0^\infty(\Omega).
\end{equation}
Where by $\langle \cdot, \cdot \rangle$ we denoted the duality pairing between $W^{-1,p}(\varpi,\Omega)$ and $W^{1,p'}(\varpi',\Omega)$.

We will need two existence and uniqueness results for problem \eqref{eq:Poisson}. The first one deals with the well posedness of \eqref{eq:Poisson} on weighted spaces and $C^1$ domains. For a proof we refer the reader to \cite[Theorem 2.5]{MR3531368}.

\begin{theorem}[well posedness for $C^1$ domains]
\label{thm:C1domain}
Let $\Omega$ be a bounded $C^1$ domain, $p \in (1,\infty)$ and $\omega \in A_p$. Then, for every $F \in W^{-1,p}(\omega,\Omega)$ there is a unique $u \in W^{1,p}_0(\omega,\Omega)$ that is a weak solution to \eqref{eq:Poisson} and, moreover, it satisfies
\[
  \| \GRAD u \|_{\bL^p(\omega,\Omega)} \lesssim \| F \|_{W^{-1,p}(\omega,\Omega)},
\]
where the hidden constant depends on $\Omega$, $[\omega]_{A_p}$, and $p$, but it is independent of $F$.
\end{theorem}

The second result deals with the well-posedness of \eqref{eq:Poisson} on Lipschitz domains. This result can be found in \cite[Theorem 2]{MR1030819} and \cite[Theorem 0.5]{MR1331981}.

\begin{theorem}[well posedness for Lipschitz domains]
\label{thm:Lipdomains}
Let $\Omega$ be a bounded Lipschitz domain. If $p \in (p_0,p_1)$ where $p_0$ and $p_1$ are defined in \eqref{eq:defofp1}, then for every $F \in W^{-1,p}(\Omega)$ there is a unique $u \in W^{1,p}_0(\Omega)$ that is a weak solution to \eqref{eq:Poisson} and, moreover, it satisfies
\[
  \| \GRAD u \|_{\bL^p(\Omega)} \lesssim \| F \|_{W^{-1,p}(\Omega)},
\]
where the hidden constant depends on $\Omega$, and $p$, but it is independent of $F$.
\end{theorem}

We are now in position to state the well posedness of \eqref{eq:Poisson}.

\begin{theorem}[well posedness on weighted spaces for Lipschitz domains]
\label{thm:main}
Let $\Omega$ be a bounded Lipschitz domain, $p \in (p_0,p_1)$, where $p_0$ and $p_1$ are defined in \eqref{eq:defofp1}, and $\varpi \in A_p(\Omega)$. Then, for every $F \in W^{-1,p}(\varpi,\Omega)$ there is a unique $u \in W^{1,p}_0(\varpi,\Omega)$ that is a weak solution to \eqref{eq:Poisson} and, moreover, it satisfies
\begin{equation}
\label{eq:aprioriest}
  \| \GRAD u \|_{\bL^p(\varpi,\Omega)} \lesssim \| F \|_{W^{-1,p}(\varpi,\Omega)},
\end{equation}
where the hidden constant depends on $\Omega$, $[\varpi]_{A_p}$, and $p$, but it is independent of $F$.
\end{theorem}

Before proving this result, we first establish a preliminary a priori estimate.

\begin{lemma}[G\aa rding-like inequality]
\label{lem:Garding}
Let $\Omega$ be a bounded Lipschitz domain, $p \in (p_0,p_1)$, where $p_0$ and $p_1$ are defined in \eqref{eq:defofp1}, $\varpi \in A_p(\Omega)$, and $F \in W^{-1,p}(\varpi,\Omega)$. If $u \in W^{1,p}_0(\varpi,\Omega)$ is a weak solution of \eqref{eq:Poisson}, then it satisfies
\[
  \| \GRAD u \|_{\bL^p(\varpi,\Omega)} \lesssim \| F \|_{W^{-1,p}(\varpi,\Omega)} + \| u \|_{L^p(\calG)},
\]
where the hidden constant depends on $\calG$, $p$ and $[\varpi]_{A_p}$, but it is independent of $F$.
\end{lemma}
\begin{proof}
Let $u_i = u \psi_i \in W^{1,p}_0(\varpi,\Omega_i)$ and $\varphi \in C_0^\infty(\Omega_i)$ then
\begin{equation}
\label{eq:keyest}
\begin{aligned}
  \int_{\Omega_i} \GRAD u_i \GRAD \varphi &= \int_{\Omega_i} \GRAD u \GRAD\left(\psi_i  \varphi \right) - \int_{\Omega_i} \varphi \GRAD u \GRAD \psi_i + \int_{\Omega_i} u \GRAD \psi_i \GRAD \varphi \\
  &= \int_{\Omega_i} \GRAD u \GRAD\left(\psi_i  \varphi \right) + \int_\calG u \DIV\left( \varphi \GRAD \psi_i \right) + \int_\calG u \GRAD \psi_i \GRAD \varphi,
\end{aligned}
\end{equation}
where we used that $\supp \GRAD \psi_i \subset \bar\calG$. This identity shows that $u_i$ is a weak solution to \eqref{eq:Poisson} over $\Omega_i \in C^{1,1}$ with right hand side $F_i$ defined by
\[
  \langle F_i, \varphi \rangle := \langle F, \psi_i\varphi \rangle + \int_\calG u \DIV\left( \varphi \GRAD \psi_i \right) + \int_\calG u \GRAD \psi_i \GRAD \varphi.
\]
Consequently, invoking the estimate of Theorem~\ref{thm:C1domain} we can obtain that
\[
  \| \GRAD u_i \|_{\bL^p(\varpi,\Omega_i)} \lesssim \| F_i \|_{W^{-1,p}(\varpi,\Omega_i)}.
\]
Now, using the fact that $\varpi$, when restricted to $\calG$ is uniformly positive and bounded we can estimate
\begin{align*}
  \| F_i \|_{W^{-1,p}(\varpi,\Omega_i)} &\lesssim \| F \|_{W^{-1,p}(\varpi,\Omega)}
    + \sup_{0 \neq \varphi \in W^{1,p'}_0(\varpi',\Omega_i)} \frac{ \int_{\calG} |u| |\GRAD \varphi| }{ \| \GRAD \varphi \|_{\bL^{p'}(\varpi',\Omega_i)}} \\
    &+ \sup_{0 \neq \varphi \in W^{1,p'}_0(\varpi',\Omega_i)} \frac{ \int_{\calG} |u| |\varphi| }{ \| \GRAD \varphi \|_{\bL^{p'}(\varpi',\Omega_i)}} \\
    &\lesssim \| F \|_{W^{-1,p}(\varpi,\Omega)} + \| u \|_{L^p(\calG)}.
\end{align*}
Combining the previous two bounds allows us to conclude
\begin{equation}
\label{eq:interiorestimate}
  \| \GRAD u_i \|_{\bL^p(\varpi,\Omega_i)} \lesssim \| F \|_{W^{-1,p}(\varpi,\Omega)} + \| u \|_{L^p(\calG)}.
\end{equation}

Define now $u_\partial = u \psi_\partial \in W^{1,p}_0(\calG)$. Similar computations, but using now Theorem~\ref{thm:Lipdomains} for the Lipschitz domain $\calG$ allow us to conclude
\[
  \| \GRAD u_\partial \|_{\bL^p(\calG)} \lesssim \| F \|_{W^{-1,p}(\varpi,\Omega)} + \| u \|_{L^p(\calG)}
\]
so that, using the uniform boundedness and positivity of $\varpi$ over $\calG$ we conclude
\begin{equation}
\label{eq:layerestimate}
  \| \GRAD u_\partial \|_{\bL^p(\varpi,\calG)} \lesssim \| F \|_{W^{-1,p}(\varpi,\Omega)} + \| u \|_{L^p(\calG)}.
\end{equation}

Since $u = u_i + u_\partial$, an application of the triangle inequality, and estimates \eqref{eq:interiorestimate} and \eqref{eq:layerestimate} yield the desired bound.
\end{proof}

We are now in position to begin proving Theorem~\ref{thm:main} with the uniqueness result.

\begin{lemma}[uniqueness]
\label{lem:unique}
Let $\Omega$ be a bounded Lipschitz domain, $p \in [2,p_1)$, where $p_1$ is defined in \eqref{eq:defofp1}, and $\varpi \in A_p(\Omega)$. If $u \in W^{1,p}_0(\varpi,\Omega)$ solves \eqref{eq:Poisson} with $F = 0$, then $u = 0$.
\end{lemma}
\begin{proof}
We begin by observing that the assumptions imply that $u \in W^{2,r}(\Omega_i)$ for every $r \in (1,\infty)$, \cite[Theorem 9.15]{MR737190}; notice that $\partial \Omega_i \in C^{1,1}$. Further, similar computations to the ones that led to \eqref{eq:keyest} reveal that, for all $\varphi \in C_0^\infty(\Omega_i)$, we have
\[
  \left| \int_{\Omega_i} \GRAD u_i \GRAD \varphi \right| \lesssim \| \GRAD \varphi \|_{\bL^{r'}(\Omega_i)}
\]
where the hidden constant depends on $r$ and $u$. This shows that $\varphi \mapsto \int_{\Omega_i} \GRAD u_i \GRAD \varphi$ defines an element of $W^{-1,r}(\Omega_i)$ so that, by Theorem~\ref{thm:Lipdomains}, we obtain that $u_i \in W^{1,2}_0(\Omega_i)$.

Since we are assuming that $\varpi \in A_p(\Omega)$, and, $p \geq 2$, we also have that $u_\partial \in W^{1,p}_0(\varpi,\calG) = W^{1,p}_0(\calG) \hookrightarrow W^{1,2}_0(\calG)$ so that, to conclude
\[
  u = u_i + u_\partial \in W^{1,2}_0(\Omega).
\]
This allows us to set $\varphi = u$ in the condition to obtain that $\GRAD u = 0$ almost everywhere and, thus, $u = 0$.
\end{proof}

Having shown uniqueness we can finally prove Theorem~\ref{thm:main}.

\begin{proof}[Proof of Theorem~\ref{thm:main}]
Consider first $p \in [2,p_1)$ and assume that \eqref{eq:aprioriest} is false. If that is the case, then it is possible to find sequences $(u_k,F_k) \in W^{1,p}_0(\varpi,\Omega) \times W^{-1,p}(\varpi,\Omega)$ such that they satisfy \eqref{eq:Poisson} with $\| \GRAD u_k \|_{\bL^p(\varpi,\Omega)} = 1$, but $F_k \to 0$ in $W^{-1,p}(\varpi,\Omega)$, as $k \to \infty$. By passing to a, not relabeled, subsequence we can assume that $u_k \rightharpoonup u \in W^{1,p}_0(\varpi,\Omega)$ and that this limit satisfies \eqref{eq:Poisson} for $F = 0$, so that, by Lemma~\ref{lem:unique}, we have that $u = 0$. On the other hand, the compact embedding of $W^{1,p}_0(\varpi,\Omega)$ into $L^p(\varpi,\Omega)$ shows that $u_k \to 0$ in $L^p(\varpi,\Omega)$, so that $\| u \|_{L^p(\calG)} = 0$. Consequently, using Lemma~\ref{lem:Garding}, we have that
\[
  1 \leq \| \GRAD u \|_{\bL^p(\varpi,\Omega)} \lesssim \| u \|_{L^p(\calG)} = 0,
\]
which is a contradiction.

With the a priori estimate \eqref{eq:aprioriest} at hand we can now show existence of a solution
$u \in W^{1,p}_0(\varpi,\Omega)$, in the case $p \in [2,p_1)$, by an approximation argument. Indeed, given $F \in W^{-1,p}(\varpi,\Omega)$ we construct a sequence $F_k \in C^\infty(\Omega)$ such that $F_k \to F$ in $W^{-1,p}(\varpi,\Omega)$. Theorem~\ref{thm:Lipdomains} then guarantees the existence of a unique $u_k \in W^{1,p}_0(\Omega)$ that solves \eqref{eq:Poisson} with right hand side $F_k$. To be able to pass to the limit with \eqref{eq:aprioriest} it is then necessary to show that $u_k \in W^{1,p}_0(\varpi,\Omega)$:
\begin{enumerate}[$\bullet$]
  \item Since $\varpi \in A_p(\Omega)$, then $u_k \in W^{1,p}(\varpi,\calG)$.
  
  \item Since $\varpi \in A_p$, we invoke the \emph{reverse H{\"o}lder inequality} \cite[Theorem 5.4]{MR1800316}, and conclude the existence of $\gamma>0$ such that $\varpi^{1+\gamma} \in L^1(\Omega_i)$. Now, given that $F_k \in C^\infty(\Omega)$, we can invoke \cite[Theorem 8.10]{MR737190} to obtain that $u_k \in W^{r,2}(\Omega_i)$ with $r$ so large that, by Sobolev embedding, the right hand side of the inequality
  \[
    \int_{\Omega_i} \varpi |\GRAD u_k|^p \leq \left( \int_{\Omega_i} \varpi^{1+\gamma} \right)^{1/(1+\gamma)} \left( \int_{\Omega_i} |\GRAD u_k|^{p(1+\gamma)/\gamma} \right)^{\gamma/(1+\gamma)}
  \]
  is finite.
\end{enumerate}
This shows that $u_k \in W^{1,p}_0(\varpi,\Omega)$ and, thus, existence of a solution.

Having proved the result for $p \in [2,p_1)$, the assertion for $p \in (p_0,2)$ follows by duality.
\end{proof}

\subsection{Application. Well-posedness with Dirac sources}
\label{sec:fluff}
Let us discuss some applications of our main result. An immediate corollary is the following.
\begin{corollary}[inf--sup condition]
\label{col:infsup}
Let $\Omega$ be a bounded Lipschitz domain, $p \in (p_0,p_1)$, and $\varpi \in A_p(\Omega)$. Then, for every $v \in W^{1,p}_0(\varpi,\Omega)$ we have that
\[
  \| \GRAD v \|_{\bL^p(\varpi,\Omega)} \lesssim \sup_{ 0 \neq w \in W^{1,p'}_0(\varpi',\Omega)} \frac{ \int_\Omega \GRAD v \GRAD w }{ \| \GRAD w \|_{\bL^{p'}(\varpi',\Omega)}}
\]
where the hidden constant is independent of $v$.
\end{corollary}
\begin{proof}
Given $v \in W^{1,p}_0(\varpi,\Omega)$ we observe that $\varpi |\GRAD v|^{p-2} \GRAD v \in \bL^{p'}(\varpi',\Omega)$ so that the functional $F_v = -\DIV(\varpi |\GRAD v|^{p-2} \GRAD v) \in W^{-1,p'}(\varpi',\Omega)$ with
\[
  \| F_v \|_{W^{-1,p'}(\varpi',\Omega)} \lesssim \| \GRAD v \|_{\bL^p(\varpi,\Omega)}^{p-1}.
\]
By Theorem~\ref{thm:main} there is a unique function $w_v \in W^{1,p'}_0(\varpi',\Omega)$ that solves \eqref{eq:Poisson} with right hand side $F_v$, \ie
\[
  \int_\Omega \GRAD w_v \GRAD \varphi = \int_\Omega \varpi |\GRAD v|^{p-2} \GRAD v \GRAD \varphi, \quad \forall \varphi \in W^{1,p}_0(\varpi,\Omega),
\]
with the corresponding estimate. Thus, setting $\varphi = v$ the assertion follows.
\end{proof}

The inf--sup condition of Corollary~\ref{col:infsup} allows us to then establish the well-posedness of the Poisson problem with Dirac sources on weighted spaces.

\begin{corollary}[well-posedness]
Let $\Omega$ be a bounded Lipschitz domain and $z \in \Omega$. Then, for $\alpha \in (d-2,d)$, and $\omega_z$ defined as in \eqref{eq:defofourweight}, there is a unique $u \in W^{1,2}_0(\omega_z,\Omega)$ that is a weak solution of
\[
  -\Delta u = \delta_z \ \text{in } \Omega, \qquad u = 0 \ \text{on } \partial\Omega.
\]
\end{corollary}
\begin{proof}
Notice that, since $\alpha \in (d-2,d) \subset (-d,d)$ and $z \in \Omega$, we have that $\omega_z \in A_2(\Omega)$. In light of Corollary~\ref{col:infsup} we only need to prove then that $\delta_z \in W^{-1,2}(\omega_z,\Omega)$, but this follows from \cite[Lemma 7.1.3]{KMR} when $\alpha \in (d-2,d)$; see also \cite[Theorem 2.3]{MR3264365}. This concludes the proof.
\end{proof}

\subsection{A weighted Helmholtz decomposition on Lipschitz domains}
\label{sub:weightdec}

As the results of \cite{MR1601373,MR1896920} show, in the study of the Stokes problem \eqref{eq:StokesStrong} it is sometimes necessary to have a weighted decomposition of the spaces $\bL^p(\omega,\Omega)$, where the weight is adapted to the singularity of $\bF$. Here we show such a decomposition for a Lipschitz domain and for a weight of class $A_p(\Omega)$.

We introduce some notation. For $p \in (1,\infty)$ and a weight $\omega \in A_p$, the space of solenoidal functions is
\[
  \bL^p_{\sigma,N}(\omega,\Omega) = \left\{ \bv \in \bL^p(\omega,\Omega): \DIV \bv = 0\right\}.
\]
The space of gradients is
\[
  \bG^p_D(\omega,\Omega) = \left\{ \GRAD v : v \in W^{1,p}_0(\omega,\Omega) \right\}.
\]
We wish to show the decomposition
\begin{equation}
\label{eq:Helmholtz}
  \bL^p(\omega,\Omega) = \bL^p_{\sigma,N}(\omega,\Omega) \oplus \bG^p_D(\omega,\Omega)
\end{equation}
with a continuous projection $\calP_{p,\omega}: \bL^p(\omega,\Omega) \to \bL^p_{\sigma,N}(\omega,\Omega)$ such that $\ker \calP_{p,\omega} = \bG^p_D(\omega,\Omega)$.

\begin{corollary}[weighted Helmholtz decomposition I]
\label{cor:Hdec}
Let $\Omega$ be a bounded Lipschitz domain, $p \in (p_0,p_1)$, where the range of exponents is defined in \eqref{eq:defofp1}. Then, if $\omega \in A_p(\Omega)$ the decomposition \eqref{eq:Helmholtz} holds.
\end{corollary}
\begin{proof}
Let $\bef \in \bL^p(\omega,\Omega)$. By Theorem~\ref{thm:main} there is a unique $u \in W^{1,p}_0(\omega,\Omega)$ that solves \eqref{eq:Poisson} with $F = \DIV \bef$. Setting $\bef = (\bef - \GRAD u) + \GRAD u$ gives, by uniqueness and the estimate on $\GRAD u$, the desired decomposition.
\end{proof}

\subsection{The Neumann problem}
\label{sub:Neumann}

We briefly comment that, with the same techniques, our result can be transferred to the case of Neumann boundary conditions. For that, all that is needed is the analogues to Theorems~\ref{thm:C1domain} and \ref{thm:Lipdomains} to carry out our considerations.

\begin{theorem}[well-posedness of the Neumann problem in Lipschitz domains]
\label{thm:Neumann}
Let $\Omega$ be a bounded Lipschitz domain, $p \in (p_0,p_1)$, with $p_0,p_1$ defined in \eqref{eq:defofp1}, and $\varpi \in A_p(\Omega)$. Then, for every $\bef \in \bL^p(\varpi,\Omega)$ there is a unique $u \in W^{1,p}(\varpi,\Omega)/\R$ such that
\[
  \int_\Omega \GRAD u \GRAD \varphi = \int_\Omega \bef \GRAD \varphi, \quad \forall \varphi \in W^{1,p'}(\varpi,\Omega)
\]
with the estimate
\[
  \| \GRAD u \|_{\bL^p(\varpi,\Omega)} \lesssim \| \bef \|_{\bL^p(\varpi,\Omega)},
\]
where the hidden constant depends on $\Omega$, $[\varpi]_{A_p}$ and $p$, but it is independent of $\bef$.
\end{theorem}
\begin{proof}
All that is needed are the analogues of Theorems~\ref{thm:C1domain} and \ref{thm:Lipdomains} to be able to proceed as before. For that, we use \cite[Theorem 3]{MR1896920} and \cite[Theorem 2]{MR1030819}, respectively.
\end{proof}

This immediately allows us to obtain a different Helmholtz decomposition, where we exchange the boundary conditions from the space of gradients into the space of solenoidal fields. Indeed,  if we define
\[
  \bL^p_{\sigma,D}(\omega,\Omega) = \left\{ \bv \in \bL^p(\omega,\Omega): \DIV \bv = 0, \bv \cdot \bn = 0 \right\},
\]
where we denote by $\bn$ the outer normal to $\Omega$ and 
\[
  \bG^p_N(\omega,\Omega) = \left\{ \GRAD v : v \in W^{1,p}(\omega,\Omega) \right\},
\]
then we can assert the following.

\begin{corollary}[weighted Helmholtz decomposition II]
In the setting of Theorem~\ref{thm:Neumann} we have the following decomposition
\begin{equation}
\label{eq:HelmholtzII}
  \bL^p(\omega,\Omega) = \bL^p_{\sigma,D}(\omega,\Omega) \oplus \bG^p_N(\omega,\Omega).
\end{equation}
\end{corollary}
\begin{proof}
Repeat the proof of Corollary~\ref{cor:Hdec} but using now Theorem~\ref{thm:Neumann}.
\end{proof}

\section{The Stokes problem}
\label{sec:Stokes}
With techniques similar to the ones used to prove Theorem~\ref{thm:main} we can prove the well-posedness of the Stokes problem \eqref{eq:StokesStrong} with singular data $\bF$ and $g$. We begin by remarking that, owing to the boundary conditions on $\bu$, we must necessarily have
\[
  \int_\Omega g = 0.
\]
Thus our notion of weak solution will be the following. For $p \in (1,\infty)$ and $\varpi \in A_p$, given $\bF \in \bL^p(\varpi,\Omega)$ and $g \in L^p(\varpi,\Omega)/\R$ we seek for a pair $(\bu, \pi) \in \bW^{1,p}_0(\varpi,\Omega) \times L^p(\varpi,\Omega)/\R$ such that for all $(\bphi,q) \in \bC_0^\infty(\Omega) \times C_0^\infty(\Omega)$ we have
\begin{equation}
\label{eq:Stokes}
    \int_\Omega \left( \GRAD \bu \GRAD \bphi - \pi \DIV \bphi \right) = \int_\Omega \bF \GRAD \bphi, \quad \int_\Omega \DIV \bu q = \int_\Omega gq.
\end{equation}

In order to derive the well--posedness of the Stokes problem \eqref{eq:Stokes} with singular data $\bF$ and $g$ we will need two auxiliary results. The first one deals with its well-posedness on weighted spaces and $C^1$ domains. For a proof of this result we refer the reader to \cite[Lemma 3.2]{MR3582412}.

\begin{theorem}[well posedness of Stokes for $C^1$ domains]
\label{thm:C1domain_Stokes}
Let $\Omega$ be a bounded $C^1$ domain, $p \in (1,\infty)$ and $\omega \in A_p$. Then, for every $\bF \in \bL^p(\omega,\Omega)$ and $g \in L^p(\omega,\Omega)/\R$ there is a unique $(\bu,\pi) \in \bW^{1,p}_0(\omega,\Omega) \times L^p(\omega,\Omega)/\R$ that is a weak solution to \eqref{eq:Stokes} and, moreover, it satisfies
\[
  \| \GRAD \bu \|_{\bL^p(\omega,\Omega)} + \| \pi \|_{L^p(\omega,\Omega)/\R} \lesssim \| \bF \|_{\bL^p(\omega,\Omega)} + \| g \|_{L^p(\omega,\Omega)},
\]
where the hidden constant depends on $\Omega$, $[\omega]_{A_p}$, and $p$, but it is independent of the data $\bF$ and $g$.
\end{theorem}

The second second result previously mentioned deals with the well-posedness of the Stokes problem \eqref{eq:Stokes} when $\Omega$ is a Lipschitz domain. As in the case of the Poisson problem it is necessary now to restrict the range of exponents $p$. However, to our knowledge, the optimal range is not available and we refer the reader to \cite[Theorem 1.1.5]{MR2987056} for a proof the following result and Figure 1 of this reference for a depiction of the allowed range of exponents for $d=2$ and $d=3$.

\begin{theorem}[well posedness of Stokes for Lipschitz domains]
\label{thm:Lipdomains_Stokes}
Let $\Omega$ be a bounded Lipschitz domain. There exists $\varepsilon = \varepsilon(d,\Omega) \in (0,1]$ such that if $|p-2| < \varepsilon$, then for every $\bF \in \bL^p(\Omega)$ and $g \in L^p(\Omega)/\R$ there is a unique  $(\bu,\pi) \in \bW^{1,p}_0(\Omega) \times L^p(\Omega)/\R$ that is a weak solution to \eqref{eq:Stokes}. In addition, this solution satisfies
\[
  \| \GRAD \bu \|_{\bL^p(\Omega)} + \| \pi \|_{L^p(\Omega)/\R} \lesssim \| \bF \|_{\bL^p(\Omega)} + \| g \|_{L^p(\Omega)},
\]
where the hidden constant depends on $\Omega$, and $p$, but it is independent of $F$.
\end{theorem}

The well-posedness for the Stokes problem is then as follows.

\begin{theorem}[Stokes problem]
\label{thm:Stokes}
Let $\Omega$ be a bounded Lipschitz domain, let $\varepsilon$ be as in Theorem~\ref{thm:Lipdomains_Stokes}, $p\in[2,2+\varepsilon)$, and $\varpi \in A_p(\Omega)$. If $\bF \in \bL^p(\varpi,\Omega)$ and $g \in L^p(\varpi,\Omega)/\R$, then there is a unique weak solution $(\bu,\pi) \in \bW^{1,p}_0(\varpi,\Omega) \times L^p(\varpi,\Omega)/\R$ of \eqref{eq:Stokes} which satisfies
\begin{equation}
\label{eq:aprioristokes}
  \| \GRAD \bu \|_{\bL^p(\varpi,\Omega)} + \| \pi \|_{L^p(\varpi,\Omega)/\R} \lesssim \| \bF \|_{\bL^p(\varpi,\Omega)} + \| g \|_{L^p(\varpi,\Omega)},
\end{equation}
where the hidden constant is independent of $\bF$.
\end{theorem}
\begin{proof}
The proof will follow the same steps as the case of the Poisson problem:
\begin{enumerate}[$\bullet$]
  \item \emph{G\aa rding inequality}: We prove that if $(\bu,\pi) \in \bW^{1,p}_0(\varpi,\Omega) \times L^p(\varpi,\Omega)/\R$ solves \eqref{eq:Stokes}, then we have
  \begin{multline}
    \label{eq:GardingStokes}
    \| \GRAD \bu \|_{\bL^p(\varpi,\Omega)} + \| \pi \|_{L^p(\varpi,\Omega)} \lesssim \| \bF \|_{\bL^p(\varpi,\Omega)} + \| g \|_{L^p(\varpi,\Omega)} \\+ \| \bu \|_{\bL^p(\calG)} + \| \pi \|_{W^{-1,p}(\varpi,\Omega_i)} + \| \pi \|_{W^{-1,p}(\calG)}.
  \end{multline}
  Indeed, by using the cutoff function $\psi_i$ and defining $\bu_i:= \bu \psi_i$ and $\pi_i:= \pi \psi_i$, we observe that $(\bu_i,\pi_i) \in \bW^{1,p}_0(\varpi,\Omega_i) \times L^p(\varpi,\Omega_i)$ solve \eqref{eq:Stokes} with
  \begin{align*}
    \int_{\Omega_i} \bF_i \GRAD \bphi &= 
    \int_\Omega \bF \GRAD \bphi 
    + \int_\calG \bu \otimes \GRAD \psi_i \GRAD \bphi 
    + \int_\calG \bu \DIV( \GRAD \psi_i \otimes \bphi ) 
    + \int_\calG \pi \bphi \GRAD\psi_i, \\
    \int_{\Omega_i} g_i q &= \int_\Omega g \psi_i q + \int_\calG \bu \GRAD\psi_i q,
  \end{align*}
  where $\bphi \in \bC_0^{\infty}(\Omega_i)$ and $q \in C_0^{\infty}(\Omega_i)$. Consequently, the estimates of \cite[Lemma 3.2]{MR3582412} yield that
  \[
    \| \GRAD \bu_i \|_{\bL^p(\varpi,\Omega_i)} + \| \pi_i \|_{L^p(\varpi,\Omega_i)} \lesssim \| \bF_i \|_{\bL^p(\varpi,\Omega_i)} + \| g_i \|_{L^p(\varpi,\Omega_i)}
  \]
  with
  \[
    \| g_i \|_{L^p(\varpi,\Omega_i)} = \sup_{0 \neq q \in C_0^\infty(\Omega_i)} \frac{ \int_{\Omega_i} g_i q }{ \| q \|_{L^{p'}(\varpi',\Omega_i)}} \lesssim 
    \| g \|_{L^p(\varpi,\Omega)} + \| \bu \|_{\bL^p(\calG)}
  \]
  and
  \begin{align*}
    \| \bF_i \|_{\bL^p(\varpi,\Omega_i)} &\lesssim \| \bF \|_{\bL^p(\varpi,\Omega)} + \| \bu \|_{\bL^p(\calG)} 
    + \sup_{0 \neq \bphi \in \bC_0^\infty(\Omega_i)} \frac{ \int_\calG \pi \bphi \GRAD \psi_i }{ \| \GRAD \bphi \|_{\bL^{p'}(\varpi',\Omega_i)}} \\
    &\lesssim \| \bF \|_{\bL^p(\varpi,\Omega)} + \| \bu \|_{\bL^p(\calG)}  + \| \pi \|_{W^{-1,p}(\varpi,\Omega_i)}.
  \end{align*}

  We now use the cutoff function $\psi_\partial$ to define the functions $\bu_\partial = \bu \psi_\partial \in \bW^{1,p}(\calG)$ and $\pi_\partial = \pi \psi_\partial \in L^p(\calG)$. A similar calculation, together with \cite[Theorem 2.9]{MR1386766} gives then the desired bound for $(\bu_\partial,\pi_\partial)$ and, thus, \eqref{eq:GardingStokes}.
  
  \item \emph{Uniqueness}: We now prove that $\bF = \boldsymbol{0}$ and $g=0$ imply $\bu = \boldsymbol{0}$ and $\pi = 0$. The argument is similar to Lemma~\ref{lem:unique}. We first observe that, by \cite[Theorem IV.4.2]{MR2808162} we have $(\bu_i,\pi_i) \in \bW^{2,r}(\Omega_i) \times W^{1,r}(\Omega_i) \hookrightarrow \bW^{1,2}(\Omega_i) \times L^2(\Omega_i)$. In addition $(\bu_\partial,\pi_\partial) \in \bW^{1,p}(\varpi,\calG) \times L^p(\varpi,\calG) \hookrightarrow \bW^{1,2}(\calG) \times L^2(\calG)$.
  
  \item \emph{A priori estimate \eqref{eq:aprioristokes}}: This is, once again, proved by contradiction. We assume \eqref{eq:aprioristokes} is false so that exist sequences
  \[
   ( \bu_k, \pi_k) \in \bW_0^{1,p}(\varpi,\Omega) \times L^p(\varpi,\Omega)/ \R, \qquad ( \bF_k, g_k) \in \bL^p(\varpi,\Omega) \times L^p(\varpi,\Omega)/\R
  \]
  such that
  $\|\GRAD \bu_k\|_{\bL^p(\varpi,\Omega)} + \| \pi_k\|_{L^p(\varpi,\Omega)} = 1$ but that $\| \bF_k \|_{L^p(\varpi,\Omega)} + \| g_k \|_{L^p(\varpi,\Omega)} \to 0$. Extracting weakly convergent subsequences and using uniqueness we conclude that the limits must be $\bu = \boldsymbol{0}$ and $\pi = 0$. However, by compactness and \eqref{eq:GardingStokes}
  \[
    1 \lesssim \| \bu \|_{\bL^p(\calG)} + \| \pi \|_{W^{-1,p}(\varpi,\Omega_i)} + \| \pi \|_{W^{-1,p}(\calG)} = 0,
  \]
  which is a contradiction.
  
  \item \emph{Existence}: Finally, we construct a solution by approximation. For that, it suffices to invoke the interior regularity of \cite[Theorem IV.4.2]{MR2808162}.
\end{enumerate}

This concludes the proof.
\end{proof}

\section*{Acknowledgements}
AJS would like to thank his colleague Tadele Mengesha for fruitful discussions.

\bibliographystyle{siamplain}
\bibliography{biblio}
\end{document}